\documentclass[11pt]{amsart} 
\usepackage{amssymb,amsmath,latexsym,enumerate,xcolor,graphicx,bbm,mathptmx,microtype,cite}
\allowdisplaybreaks

\theoremstyle{plain}
\newtheorem{thm}{Theorem}[section]
\newtheorem{lem}[thm]{Lemma}
\newtheorem{cor}[thm]{Corollary}
\newtheorem{prop}[thm]{Proposition}

\theoremstyle{definition}
\newtheorem{defin}[thm]{Definition}

\theoremstyle{remark}
\newtheorem*{rem}{Remark}

\newtheorem{example}[thm]{Example}

\newcommand \R{\mathbb{R}}
\newcommand \Z{\mathbb{Z}}

\newcommand \HH{\mathbb{H}}
\newcommand \minrep{\mathrm{minrep}}
\newcommand \pmax{\mathrm{pmax}}
\newcommand \nmax{\mathrm{nmax}}
\newcommand \ehr{\mathrm{ehr}}
\newcommand \conv{\mathrm{conv}}
\newcommand \natdes{\mathrm{natdes}}
\newcommand \JH{\mathrm{JH}}
\newcommand \id{\mathrm{id}}
\newcommand \PLC[1]{\overline{#1}^\mathrm{PLC}}
\newcommand \op[1]{\mathcal{O}_#1}
\newcommand \ok[1]{\mathcal{K}_#1}
\newcommand \Hom[1]{\mathrm{Hom}(#1)}

\newcommand{\df}[1]{\textit{#1}}

\begin{document}

\title{Signed Poset Polytopes}

\author{Matthias Beck}
\address{Department of Mathematics\\
         San Francisco State University\\
         San Francisco, CA 94132\\
         U.S.A.}
\email{becksfsu@gmail.com}

\author{Max Hlavacek}
\address{Department of Mathematics \& Statistics, Pomona College, Claremont, CA 91711, U.S.A.}
\email{max.hlavacek@pomona.edu}


\begin{abstract}
Stanley introduced in 1986 the order polytope and the chain polytope for a given finite
poset. These polytopes contain much information about the poset and have given rise to
important examples in polyhedral geometry. In 1993, Reiner introduced signed posets as
natural type-B analogues of posets. We define and study signed order and chain polytopes.
Our results include convex-hull and halfspace descriptions, unimodular triangulations,
Ehrhart $h^*$-polynomials and their relations to signed permutation statistics, and a
Gorenstein characterization of signed order and chain polytopes.
\end{abstract}



\date{7 November 2023}

\maketitle


\section{Introduction}

In the seminal paper~\cite{S86}, Stanley introduced two geometric incarnations
of a given finite partially ordered set (\emph{poset}) $\Pi$: 

\begin{defin}
The \emph{order polytope} of $\Pi$ is given by
\[\mathcal{O}(\Pi) \ := \ \left\{x \in \mathbb{R}^\Pi : \, 0 \leq x_p \leq 1 \mbox{ for all } p \in \Pi \mbox{ and } x_a \leq x_b \mbox{ when } a\leq_\Pi b \right\}.\]  
\end{defin}

\begin{defin}
The \emph{chain polytope} of $\Pi$ is given by 
\[
  \mathcal{C}(\Pi) \ := \ \left\{ x \in \mathbb{R}^\Pi : \, x_p \ge 0 \mbox{ for
all } p \in \Pi \mbox{ and } x_{c_1} + \dots + x_{c_k} \leq 1 \mbox{ for every
chain } c_1 <_\Pi \dots <_\Pi c_k \right\} .
\]
\end{defin}

There is also a natural unbounded conical analogue of an order polytope:

\begin{defin}
The \emph{order cone} of $\Pi$ is given by
\[\mathcal{K}(\Pi) \ := \ \left\{x \in \mathbb{R}^\Pi : \, 0 \leq x_p  \mbox{
for all } p \in \Pi \mbox{ and } x_a \leq x_b \mbox{ when } a\leq_\Pi b \right\}.\]  
\end{defin}

These polyhedra contain much information of the given poset, e.g., about its
filters, chains, and linear extensions.
Conversely, order polytopes (and, to a lesser extent, chain polytopes) have
given a fertile ground in polyhedral and discrete geometry as a class of
$0/1$-polytopes with many, sometimes extreme, at other times only conjectured,
properties; see, e.g., \cite[Chapter~6]{B18}.

It is a short step (which we will detail in Section~\ref{sec:defs} below) to think
about a given partial order on $[n] := \{ 1, 2, \dots, n \}$ as a certain subset
of the type-$A$ root system $A_n := \{e_i-e_j: 1\leq i<j\leq n\}$, where $e_j$ is the $j$th unit vector in $\R^n$.
Here we need the following notion of the positive linear closure of a subset of a root system:
 
 \begin{defin}
Let $\Phi$ be a root system.  For a subset $S \subset \Phi$, let $\PLC{S}$ be the set of positive linear combinations of elements in~$S$.
\end{defin}

In \cite{R93}, Reiner generalised the Coxeter description of a classical poset
to type $B$ and named these objects \emph{signed posets}.  

\begin{defin}
Let $B_n$ be the root system $\{\pm e_i: 1 \leq i \leq n\} \cup \{\pm e_i \pm e_j: 1 \leq i <j\leq n\}$.  A signed poset $P$ (on $n$ elements) is a subset $P \subset B_n$ satisfying
\begin{enumerate}
    \item $\alpha \in P$ implies $-\alpha \not\in P$;
    \item $\PLC{P} = P$.
\end{enumerate}
\end{defin}

Reiner proved several type-$B$ analogues of classical posets results, e.g., on
order ideals, permutation statistics, and $P$-partitions. While \cite{R93} has
given rise to further work in algebraic combinatorics, the analogous polyhedral
constructions seem to not have been introduced/studied. (An exception of sorts
is~\cite{S08}, and we indicate its relation to the present article below.)
Our goal is to remedy this situation
and explore the geometry of signed posets. We define order polytopes and
order cones for signed posets, modeled after the Coxeter-group descriptions of these
objects for classical posets.
 
\begin{defin}\label{def:signedgeo}
Let $P$ be a signed poset on $[n]$. Then
    \[\ok{P} \ := \ \left\{x \in \R^n: \langle\alpha, x\rangle \geq 0 \mbox{ for all } \alpha \in
P\right\}\]
and 
\[\op{P} \ := \ \left\{x \in \R^n:\langle \alpha, x\rangle \geq 0 \mbox{ for all } \alpha \in
P\right\} \cap [-1,1]^n.\]
\end{defin} 

In Section~\ref{sec:defs} we give detailed connections between (signed) posets,
Coxeter groups, bidirected graphs, permutation statistics, and the resulting polyhedral objects.
Section~\ref{sec:altdes} contains convex-hull and halfspace description of
signed order polytopes, as well as canonical triangulations.
In Section~\ref{sec:hstar} we compute the Ehrhart $h^*$-polynomial of a signed
polytope, encoding the integer-point structure in dilates of the polytope.
Analogous to classical posets, this is related to permutation statistics (now of
type $B$).
Section~\ref{sec:gorenstein} gives a characterization of signed order polytopes
that are Gorenstein; as a consequence these polytopes have a symmetric and unimodal $h^*$-polynomial.
Finally, in Section~\ref{sec:chainpoly} we propose one definition of a signed
chain polytope and study its properties, giving yet another new class of
Gorenstein polytopes.


\section{Signed Posets, Their Cones, and Their Polytopes}\label{sec:defs}
We start by detailing the interpretation of a poset in terms of type-$A$ root
systems, as introduced by Reiner~\cite{R93}.

\begin{prop} \label{prop:coxposet} 
 Let $f$ be the following map from the set of posets on $[n]$ to subsets of $A_n$:
\[f(\Pi) \ := \ \{e_j - e_i: i <_\Pi j\}\]
Then $f$ gives a bijection between partial orders 
of $[n]$ and subsets $P \subseteq A_n$ 
satisfying
\begin{enumerate}[{\rm (1)}]
    \item $\alpha \in P$ implies $-\alpha \not\in P$;
    \item $\PLC{P} = P$.
\end{enumerate}
\end{prop}

The first property comes from the antisymmetry property of a poset, and the second from
the transitivity property.

\begin{rem}\label{rem: order-reversing}
In several works, including \cite{R93}, the map from posets to subsets of $A_n$
is defined as $f(\Pi) = \{e_i - e_j: i <_\Pi j\}$. We switched the direction of
the inequality to make our upcoming description of order cones consistent with
the description of $P$-partitions found in~\cite{R93}.  
\end{rem}

Classically-defined order cones and order polytopes can be reformulated via the Coxeter group description of the poset.
Let $\Pi$ be a poset 
and recall that $f(\Pi) =  \{e_j - e_i: i <_\Pi j\}$. Then
\[\mathcal{K}(\Pi) \ = \ \left\{x \in \R_{\geq 0}^\Pi: \, \langle \alpha, x\rangle\geq 0 \ \text{
for all } \ \alpha \in f(\Pi) \right\}\]
and 
\[\mathcal{O}(\Pi) \ = \ \left\{x \in \R_{\geq 0}^\Pi: \, \langle\alpha, x\rangle \geq 0 \ \text{ for
all } \ \alpha \in f(\Pi) \right\} \cap [0,1]^\Pi.\]

We often visually view signed posets as  bi-directed graphs, using the definitions found in \cite{Z91}. 

\begin{defin}
Let $\Gamma = (V, E)$ be a graph.  The \emph{incidence set} $I(\Gamma)$ of $\Gamma$ consists of all pairs $(e,v)$ for $v \in e$.
\end{defin}

\begin{defin}
A \emph{bidirected graph} is a graph $\Gamma$ together with a \emph{bidirection} $\sigma$,
which is defined as any map $I(\Gamma) \to \{+, -\}$.   
\end{defin}

We view a signed poset $P$ on $[n]$ as a bidirected graph with vertex set $[n]$: the
elements of $P$ are the edges of the bidirected graph, and the bidirection $\sigma$ is defined so that 
\begin{itemize}
\item $e_j$ corresponds to a loop $e$ on $j$ with $\sigma(e,j) = +$, 
\item $-e_j$ corresponds to a loop $e$ on $j$ with $\sigma(e,j) = -$, 
\item $e_i + e_j$ corresponds to an edge $e=ij$ with $\sigma(e,i) = \sigma(e,j) = +$, 
\item $-e_i-e_j$ corresponds to an edge $e=ij$ with $\sigma(e,i) = \sigma(e,j) = -$, 
\item $e_i - e_j$ corresponds to an edge $e=ij$ with $\sigma(e,i) = +$ and $\sigma(e,j) = -$, and 
\item $-e_i + e_j$ corresponds to an edge $e=ij$ with $\sigma(e, i) = -$ and $\sigma(e,j)=+$.  
\end{itemize}

\begin{example}\label{ex:bidi minrep}
The left image in Figure \ref{fig:bidi} shows a bidirected graph representing the signed
poset $P = \{e_1 + e_2, -e_1 + e_2, e_2\}$. A discussion in \cite[p.~329]{R93} states that
every signed poset has a unique minimal representation, i.e., a minimal subset whose
positive linear closure is the whole signed poset.  However, finding this minimal
representation is not as straightforward as finding the cover relations of a classical
poset.  The minimal representation of $P$ is shown in the right image in Figure~\ref{fig:bidi}. 

\begin{figure}[ht]
	\centering
        \includegraphics[width=.6\textwidth]{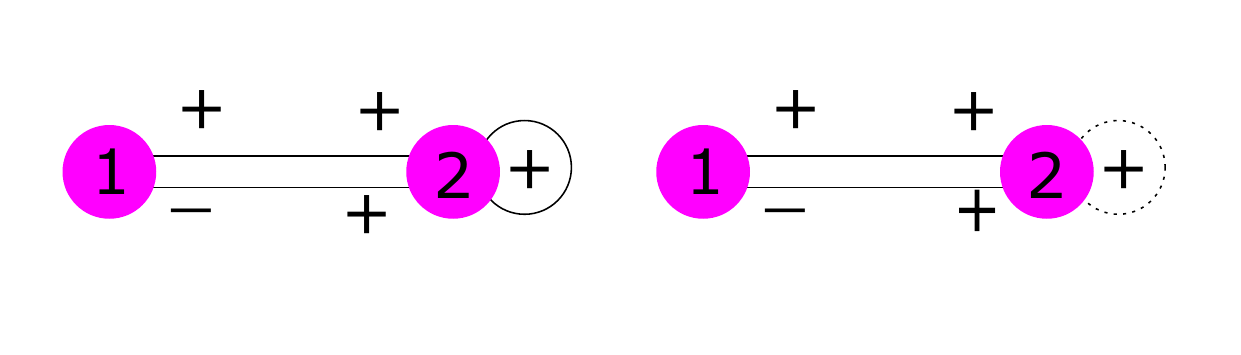}
        \caption{The left hand side shows the bidirected graph representation of $P =  \{-e_1+e_2, e_1+e_2, e_2\}$.  The right hand side indicates that the unique minimal representation of $P$ is $\{-e_1 + e_2,e_1 + e_2\}$. Since $e_2 = \frac{1}{2} (-e_1 + e_2) + \frac{1}{2}(e_1+e_2)$, we see that $e_2$ is in the positive linear closure of the other two elements, and thus is not in the minimal representation of $P$.  This is indicated on the right hand side by using a dotted loop instead of a solid loop to represent $e_2$.  }
\label{fig:bidi}
	\end{figure}

%
%
%

\end{example}

Reiner also gave a notion of homomorphism of signed posets in \cite{R93}. Before we
discuss this definition, we establish some facts and definitions pertaining to $S_n^B$,
mostly following the notation set in \cite{bb13}.

\begin{defin}
A \emph{signed permutation} on $[n]$ is a bijection $\omega$ on $\{ \pm 1, \pm 2, \dots, \pm n \}$ such that $\omega(-i) = -\omega(i)$ for all $i \in [n]$.  We refer to the group of all such bijections as $S_n^B$, the signed permutation group on $[n]$, where the group operation is composition.  
\end{defin}

We will use one-line notation for an element $\omega \in S_n^B$;
later we will also use the notation $\omega = (\pi, \epsilon)$ where $\pi \in S_n$ is defined
by $\pi_i := |\omega(i)|$ and $\epsilon \in \{1,-1\}^n$ is defined via $\epsilon_i := \mbox{sign}(\omega(i))$.
We will also use the following fact.

\begin{prop}\label{prop: b generators}
The set $\{s_1, \dots, s_{n-1}, s_0\}$ generates $S_n^B
$, where $s_i := [1, \dots, i-1, i+1, i, \dots, n]$ for $i \in [n-1]$ and $s_0 := [-1,2,\dots, n]$
\end{prop}


The elements of $S_n^B$ have a natural action on the type-$B$ root system.  
\begin{defin}\label{group action}
	  The elements of $S_n^B$ have a linear action on $B_n$ generated as follows,
where $i,j \in [n]$.  If $\omega(i) = j$, then $\omega e_i = e_j$.  If $\omega(i) = -j$, then $\omega e_i = -e_j$.  
\end{defin}
 
\begin{defin}\label{def:signedhomomorph}
Let $P_1$ and $P_2$ be signed posets on $[n]$.  Then $P_1$ and $P_2$ are \emph{isomorphic} if there exists $\omega \in S_n^B$ such that $\omega P_1 = P_2$.  
\end{defin}

\begin{example}\label{ex:all 2d}
We see from Figure \ref{fig:examples} that there are many more combinatorial types of order polytopes in the signed poset setting than in the classical poset setting, even in two dimensions.  The order polytopes shown come from the following signed posets:

\begin{enumerate}[(A)] 
    \item $P = \{e_1, e_2\}$
    \item $ P = \{\}$
    \item $ P = \{e_2\}$
    \item $ P = \{e_1+e_2, -e_1+e_2,e_1,e_2\}$
    \item $ P = \{-e_1+e_2\}$
    \item $ P = \{-e_1+e_2, e_1+e_2, e_2\}$
    \item $ P = \{e_2, e_1+e_2\}$
\end{enumerate}
The polytope labeled (F) is the order polytope of the signed poset illustrated in Figure
\ref{fig:bidi}.  Note that the supporting hyperplanes of this polytope are given by $x_2 =
1$ and $\langle\alpha, x\rangle = 0$, where $\alpha$ is an element of the unique minimal
representation described in Example 1.6.  We will generalize this observation below. 
\end{example}


\begin{figure}[ht]
	\centering
        \includegraphics[width=.9\textwidth]{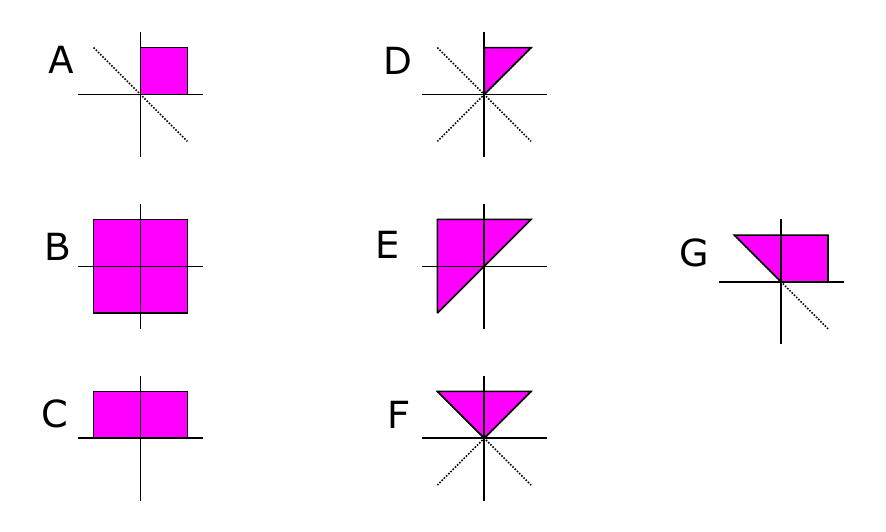}
        \caption{Some order polytopes for signed posets on two elements.}\label{fig:examples}
	\end{figure}

We note that our geometric constructions play nicely with Reiner's definition of signed poset isomorphism.

\begin{defin}\label{def: unimodular eq}
	Two lattice polytopes $Q$ and $Q'$ are \emph{unimodularly equivalent} if there is an affine lattice isomorphism of the ambient lattices mapping $Q'$ onto $Q$.
\end{defin}

\begin{prop}\label{prop: geo iso nice}
	Suppose $P$ and $P'$ are isomorphic signed posets on $[n]$. Then $\op{P}$ and $\op{P'}$ are unimodularly equivalent.  
\end{prop}

\begin{proof}
Since $P'$ is isomorphic to $P$, then $P' = \omega P$ for some $\omega \in S_n^B$.  By Proposition \ref{prop: b generators}, it suffices to verify Proposition \ref{prop: geo iso nice} for the cases in which  $\omega$ is equal to the generators $s_i = [1, \dots, i-1, i+1, i, \dots, n]$, and $s_0 = [-1, 2,\dots, n]$. 

 Suppose $\omega = s_i$ for $i \in [n-1]$.  Then 
\[\op{P} \ = \ \left\{x \in \R^n: \, \langle\alpha, x\rangle \geq 0 \mbox{ for all } \alpha \in P
\right\} \cap [-1,1]^n\]
and 
\[\op{P'} \ = \ \left\{x \in \R^n: \, \langle s_i\alpha, x \rangle\geq 0 \mbox{ for all } \alpha \in P
\right\} \cap [-1,1]^n.\]
From this description, we see that the hyperplanes defining $\op{P'}$ are reflections of the hyperplanes defining $\op{P}$ about the hyperplane $x_i = x_{i+1}$ (since the cube $[-1,1]^n$ is symmetric about this hyperplane).  This is an affine lattice isomorphism, so $\op{P}$ and $\op{P'}$ are unimodularly equivalent.  

 Similarly, for $\omega = s_0$, the hyperplanes defining $\op{P'}$ are reflections of the
hyperplanes defining $\op{P}$ about the hyperplane $x_1 = 0$, which is again an affine lattice isomorphism.  
\end{proof}

We note that the order cone of a classical poset is always pointed.  This is not the case for order cones of signed posets.  This makes it so that one cannot always write down a rational generating function for the integer point transform of such an order cone.  However, we can construct a pointed cone encoding the same information by homogenizing the order polytope, defined below following, e.g., \cite{B18}.  
As is the case for classical posets, the order polytope is related to the order cone of a signed poset as follows: 

\begin{defin}
	Let $Q$ be a polytope in $\R^n$.  The \emph{homogenization} of $Q$ is given by 
	\[\Hom{Q} \ := \ \left\{(\mathbf{x},t) \in \R^{n+1}: \, t\in \R_{\geq 0}, \
\mathbf{x}\in tQ \right\} . \]
\end{defin}

\begin{prop}\label{prop:fulldim}
For a signed poset $P$ with $n$ elements, $\hom(O_P) = K_{\hat{P}}$, where $\hat{P}$ is given by $P \cup \{e_{n+1} \pm e_i: 1\leq i \leq n\}$.  
\end{prop}

\begin{proof}
Both $\hom(O_P)$ and $K_{\hat{P}}$ are given by
\[ \left\{(\mathbf{x}, t) \in \R^{n+1}: \, -t\leq x_i \leq t \mbox{ for all } i \in [n]
\mbox{ and } \langle \alpha, \mathbf{x}\rangle \leq 0 \mbox{ for all } \alpha \in P
\right\}. \qedhere \]
\end{proof}

We now show that $\ok{P}$ (and thus $\op{P}$) is full dimensional:

\begin{prop} \label{prop: full-dim}
	Let $P$ be a signed poset on $[n]$. Then $\dim(\ok{P}) = n$.
\end{prop}

\begin{proof}
We prove that we can construct a point $\phi: [n] \to \R$ in the interior of $\ok{P}$.
We argue inductively on $n$.  The base case $n=1$ holds, since the order cone is either
the non-negative ray, the non-positive ray, or $\R$.

Now suppose we have a signed poset $P$ on $[n]$ with $n \ge 2$. Let $P'$ be the signed poset on $[n-1]$ obtained from restricting $P$ to the set $\{\pm e_i: 1 \leq i \leq n-1\} \cup \{\pm e_i \pm e_j: 1 \leq i <j\leq n-1\}$.  By our inductive hypothesis, there exists a point $\phi':[n-1] \to \R$ in the interior of $\ok{P'}$.  Our strategy is to show that we can extend $\phi'$ to a point $\phi$ in the interior of $\ok{P}$. We first let $\phi(i) = \phi'(i)$ for $1\leq i \leq n-1.$  We now need to show that there exists a choice of $\phi(n)$ such that $\phi$ is in the interior of $\ok{P}$. To make the following equations easier to look at, let $\phi(n) =x$.  

The only way in which there is no viable choice for $x$ is if the hyperplanes defining $\ok{P}$ together with the choice of $\phi'$ give rise to inequalities of the form $a<x<b$ for some $a\geq b$. 
We first list the ways that we could get a restriction of the form $x>a$:

\begin{enumerate}[(i)] 
    \item Suppose $e_n-e_j \in P$ where $j \neq n$.  Then $\phi \in \ok{P}^\circ$ implies $x>\phi(j)$.
    \item Suppose $e_n+e_j \in P$ where $j \neq n$.  Then $\phi \in \ok{P}^\circ$ implies $x>-\phi(j)$.  
    \item Suppose $e_n \in P$.  Then $\phi \in \ok{P}^\circ$ implies $x> 0$.  
\end{enumerate}

We next list the ways that we could get a restriction of the form $x<b$.  

\begin{enumerate}[(i)] 
\setcounter{enumi}{3}
    \item Suppose $-e_n+e_k \in P$ where $k \neq n$.  Then $\phi \in \ok{P}^\circ$
implies $x<\phi(k)$.
    \item Suppose $-e_n-e_k \in P$ where $k \neq n$.  Then $\phi \in \ok{P}^\circ$
implies $x<-\phi(k)$.  
    \item Suppose $-e_n \in P$.  Then $\phi \in \ok{P}^\circ$ implies $x< 0$.  
\end{enumerate}

We now show that in all cases in which we have a restriction of the form $x<b$ and a restriction of the form $x>a$, we get that $a<b$.  (And thus, there is a solution for $x$).

\begin{enumerate}
\item[Case 1:] We have situations (i) and (iv) above and thus the restriction
$\phi(j)<x<\phi(k)$.  Since $e_n-e_j, -e_n+e_k \in P$, we know that $e_k-e_j$ is in $P$
and $P'$.  Now $\phi' \in \ok{P'}^\circ$ implies $\phi(j) < \phi(k)$. 

\item[Case 2:] We have situations (i) and (v), and thus $\phi(j)<x<-\phi(k)$.  Since
$e_n-e_j,-e_n-e_k \in P$, we know that $-e_j-e_k$ is in both $P$ and $P'$.  Since $\phi'$
was chosen to be in $\ok{P'}^\circ$, we have $\phi(j) <- \phi(k)$.  

\item[Case 3:] We have the situations (i) and (vi) and thus $\phi(j)<x<0$. Since
$e_n-e_j, -e_n \in P$, we know that $-e_j \in P, P'$ and so $\phi(j) <0$. 

\item[Cases 4--8:] The arguments for the situations (ii) and (iv), (ii) and (v), (ii) and
(vi), (iii) and (iv), (iii) and (v) 
are similar as the previous three cases.  

\item[Case 9:] We have the situations (iii) and (vi).  This would imply that $e_n,-e_n \in P$, which violates the asymmetry property of a signed poset.  
\end{enumerate}

Thus, in all viable cases it is possible to choose a value of $\phi(n)$ so that $\phi \in
\ok{Pa}^\circ$.  
\end{proof}

The map $\phi(x)$ constructed above can be modified to give a signed permutation in the
Jordan--H\"{o}lder set of $P$, which will be necessary below when we describe triangulations of~$\ok{P}$. 

\begin{defin}
Let $P$ be a signed poset on $[n]$.  The \emph{Jordan--H\"{o}lder set} $\JH(P)$ of $P$  is the set of signed permutations $\omega \in S_n^B$ such that $\omega$ is order-preserving, that is, $\langle \omega, \alpha\rangle \geq 0$ for all $\alpha \in P$, where we think of $\omega$ as the point $(\omega_1, \dots , \omega_n)\in \R^n$.  
\end{defin}

\begin{rem}
Reiner gives a definition of Jordan--H\"{o}lder set in \cite{R93} that is in the same
spirit, but slightly different to the one here. He defines the Jordan--H\"{o}lder set of
a signed poset $P$ on $[n]$ as the set of all $\sigma \in B_n$ such that $\alpha \in
\sigma B_n^+$ for all $\alpha \in P$.  Here, $B_n^+$ refers to the set $\{e_i: \, 1\leq
i\leq n\}\cup \{e_i \pm e_j: \, 1\leq i\leq j\leq n\}$, the \emph{positive roots} of
$B_n$.  Our definition $\JH(P)$ can be rewritten similarly, as the set of all $\sigma \in
B_n$ such that  $\alpha \in \sigma S$, where $S=\{e_i: \, 1\leq i\leq n\}\cup \{e_i \pm
e_j: \, 1\leq j\leq i\leq n\}$ .  The difference between our definition and Reiner's is analogous to the slight differences in our early definitions discussed in Remark~\ref{rem: order-reversing}.
\end{rem}

To get from the point $\phi$ described in the proof of Proposition \ref{prop:fulldim} to
the corresponding signed permutation $\omega \in \JH(P)$, do the following:  Consider the
coordinates of $\phi$ and determine which one has the highest absolute value.  Replace
this one with $n$ or $-n$, corresponding to the sign of this coordinate.  Repeat with the
remaining coordinates, this time replacing with $n-1$ or $-(n-1)$.  Repeat until you have
replaced every coordinate.  Call this point $p$.  Then the corresponding signed
permutation is the unique element of $\omega \in S_n^B$ such that $(\omega(1),\dots, \omega(n) ) = p$.  

We now introduce the idea of a naturally labeled signed poset, mimicking a similar notion for classical posets.  

\begin{defin}
Let $P$ be a signed poset on $[n]$, and let $\id\in S_n^B$ be the identity element.  We say $P$ is \emph{naturally labeled} if and only if $\id\in\JH(P)$.
\end{defin}

\begin{prop}\label{prop: every poset nat label}
Every signed poset is isomorphic to a naturally labeled signed poset.  
\end{prop}

\begin{proof}
	Let $P$ be a signed poset.  As a consequence of the proof of Proposition
\ref{prop: full-dim}, there exists some $\omega \in \JH(P)$. Consider $\omega^{-1} P$ which, by definition, is isomorphic to
$P$.  To show that $\id \in \JH(\omega^{-1}P)$, it suffices to show that $\langle \id, \omega^{-1}\alpha \rangle \geq 0 $ for all $\alpha\in P$.  
	
	For any $\alpha$ in the type-$B$ root system and any $\omega \in S_n^B$, we have $\langle \id, \omega^{-1}\alpha\rangle = \langle \omega, \alpha \rangle$.  This can be shown by noting that $\langle \id, \omega^{-1}e_i\rangle = \langle \omega, e_i \rangle= \omega(i)$ and then using linearity.  

	Thus, for any $\alpha \in P$, $\langle \id, \omega^{-1}\alpha \rangle =
\langle \omega, \alpha \rangle \geq 0$, since $\omega \in \JH(P)$.
\end{proof}


\section{Alternative Descriptions of $\ok{P}$ and $\op{P}$}\label{sec:altdes}
We now discuss a few other descriptions of the order cones and polytopes of signed
posets.  Definition \ref{def:signedgeo}  gives a hyperplane description of $\op{P}$ and
$\ok{P}$, but this hyperplane description may be redundant. This is also true for the
definition of classical order polytopes and cones.  The following proposition
from \cite{S86} gives an irredundant representation for the order polytopes of classical posets. 

\begin{prop}[Stanley~\cite{S86}]
Let $\Pi$ be a classical poset.  Then, an irredundant representation of $\mathcal{O}(\Pi)$ is given by

\[\mathcal{O}(\Pi) \ = \ \left\{x \in \R^\Pi: \begin{array}{ll}
         x_a \leq x_b  & \text{if } a\lessdot b\\
        x_a \geq 0 & \text{if }a\text{ is a minimal element} \\
        x_a\leq 1 & \text{if }a \text{ is a maximal element}
        \end{array}\right\},\]
        where $a \lessdot_\Pi b$ means that $a \leq_\Pi b$ is a cover relation in $\Pi$.
\end{prop}

As hinted in Example \ref{ex:all 2d}, we can give an similar irredundant hyperplane representation for order polytopes of signed posets.  First, we need a few propositions and definitions. The next proposition from \cite{R93} was briefly described earlier in Example \ref{ex:bidi minrep}.

\begin{prop}[Reiner \cite{R93}]
Let $P$ be a signed poset. Then there exists a unique minimal subset of $P$, called the \emph{minimal representation} of $P$ and denoted $\minrep(P)$,  such that 
\[\PLC{\minrep(P)} \ = \ P \, .\]
\end{prop}

Note that the relations in $\minrep(P)$ are analogous to the cover relations of a classical poset, since both give an irredundant description of the (signed) poset up to transitivity.  We now give a definition for two types of maximal elements of a signed poset.  

\begin{defin}
Let $P$ be a signed poset on $[n]$. 
We say $i \in [n]$ is a \emph{positive maximal element} of $P$ if and only if every
relation adjacent to $i$ is of the form $e_i$ or $e_i \pm e_j$, i.e., the $e_i$ portion of the adjacent relations is positive. We denote the set of positive maximal elements of $P$ as $\pmax(P)$. 
We say $i \in [n]$ is a \emph{negative maximal element} of $P$ if and only if every relation adjacent to $i$ is of the form $-e_i$ or $-e_i \pm e_j$, i.e., the $e_i$ portion of the adjacent relations is negative. We denote the set of negative maximal elements of $P$ as $\nmax(P)$. 
\end{defin}

We can now give an irredundant representation of $\op{P}$.

\begin{prop}
Let $P$ be a signed poset on $[n]$.
Then an irredundant representation of $\op{P}$ is given by
\begin{equation}\label{eq:irrrep}
  \op{P} \ = \ \left\{x \in \R^n: \begin{array}{ll}
        \langle \alpha, x\rangle \geq 0 & \text{if } \alpha \in \minrep(P)\\
        x_i \leq 1 & \text{if } i \in \pmax(P) \\
        x_i \geq -1 & \text{if } i \in \nmax(P)
        \end{array}\right\}.
\end{equation}
\end{prop}

\begin{proof}
We first show that \eqref{eq:irrrep} does indeed describe $\op{P}$.  Since \eqref{eq:irrrep} consists of a subset of the inequalities listed in  Definition \ref{def:signedgeo}, it suffices to show that each inequality in Definition \ref{def:signedgeo} is implied by the inequalities listed in \eqref{eq:irrrep}. First of all, each inequality of the form $\langle\alpha, x\rangle \geq 0$  for $\alpha \in P$ is implied by the inequalities of the form $\langle \alpha, x\rangle \geq 0$ for $\alpha \in \minrep(P)$ from the definition of $\minrep(P)$. 

We now show that the inequalities $-1\leq x_i \leq 1$ for all $i \in [n]$ are implied by the inequalities listed in the proposition.  It suffices to show that for each $i \in [n]$, one of the following holds (to ensure that $x_i \leq 1$):
\begin{enumerate}[(i)]
    \item $i \in \pmax(P)$,
    \item $-e_i \in P,$ since $x_i \leq 0$ implies $x_i \leq 1$,  
    \item $-e_i +e_p \in P$, where $p \in \pmax(P)$,
    \item $-e_i -e_n \in P$, where $n \in \nmax(P)$,
\end{enumerate}
and one of the following holds (to ensure that $-1 \leq x_i$):
\begin{enumerate}[(i)]
    \setcounter{enumi}{4}
    \item $i \in \nmax(P)$,
    \item $e_i \in P,$ since $x_i \geq 0$ implies $x_i \geq -1$,  
    \item $e_i +e_p\in P$, where $p \in \pmax(P)$,
    \item $e_i -e_n \in P$, where $n \in \nmax(P)$.
\end{enumerate}

We show by induction on $n$ that this is always true. It is true for $n=1$ since in that case one of (i) or (ii) must be true, and one of (v) or (vi) must be true.   

Now, consider some $i \in P$, and let $P'$ be the signed poset on $[n]-\{i\}$ with all relations adjacent to $i$ removed (similarly as in the proof of Proposition \ref{prop:fulldim}).  We first show that $i$ satisfies (i), (ii), (iii), or (iv).  Assume that $i$ does not satisfy (i) or (ii), with the aim to show that $i$ satisfies (iii) or (iv). By our assumptions, $i$ must be adjacent to a relation of the form $-e_i + e_j$ or $-e_i - e_j$.  

\textbf{Case 1:} Assume $i$ is adjacent to a relation of the form $-e_i + e_j$.  We know by our inductive hypothesis that within $P'$, $j$ must satisfy (i), (ii), (iii), or (iv). We go through each of these subcases.
\begin{itemize}
\item Suppose $j$ satisfies (i) and   is thus an element of $\pmax(P')$.  We show that $j$ is also an element of $\pmax(P)$.  The only obstruction to this is if $-e_i -e_j$ or $e_i-e_j$ are in $P$.  However, $e_i-e_j$ cannot be in $P$ since its opposite $-e_i+e_j$ was already assumed to be in $P$.  Suppose $-e_i-e_j$ is in $P$.  Then, by transitivity with $-e_i+e_j$, we have that $-e_i$ is in $P$ which is not the case since we assumed $i$ does not satisfy (ii). Thus, $j$ is an element of $\pmax(P)$, so $i$ satisfies (iii).  
\item Now suppose $j$ satisfies (ii), so $-e_j$ is an element of $P'$ and thus $P$.  However, by transitivity with $-e_i + e_j$, $-e_i$ is in $P$, which again is not possible by our assumptions. 
\item Now suppose $j$ satisfies (iii).  Then $-e_j+e_p$ is in $P'$ and thus $P$, where $p$ is in $\pmax(P')$.  We first note that by transitivity, $-e_i + e_p$ is in $P$.  We now show that $p$ is not only in $\pmax(P')$, but also in $\pmax(P)$. The only obstructions to this are if $-e_i -e_p$ or $e_i - e_p$ are in $P$.  Note that $e_i-e_p$  cannot be in $P$ since we already established that its opposite, $-e_i+e_p$, is in $P$.  We also see that $-e_i-e_p$ cannot be in $P$ since by transitivity then $-e_i$ would be in $P$, which we assumed was not the case.  Thus, $p \in \pmax(P)$ and so $i$  satisfies (iii).  
\item Now suppose $j$ satisfies (iv).  Then $-e_j-e_n$ is in $P'$ and thus $P$, for some $n$ in $\nmax(P')$.  By transitivity, we see that $-e_i -e_n$ is in $P$.  It suffices to show that $n$ is not only in $\nmax(P')$ but also in $\nmax(P)$.  The only obstructions to this are if $-e_i +e_n$ or $e_i+e_n$ are in $P$.  We see that $e_i +e_n$ cannot be in $P$ because we already established that that its opposite $-e_i - e_n$ is. If $-e_i + e_n$ is in $P$, then by transitivity with $-e_i - e_n$, $-e_i$ is in $P$, which we assumed was not the case.  So $n$ is in $\nmax(P)$ and thus $i$ satisfies (iv).  
\end{itemize}

\textbf{Case 2:} Assume $i$ is a adjacent to a relation of the form $-e_i -e_j$.  We know by our inductive hypothesis that within $P'$, $j$ must satisfy (v),(vi),(vii), and (viii).  We go through each of these subcases and show that each case implies $i$ satisfies (iii) of (iv).  

\begin{itemize}
    \item Suppose $j$ satisfies (v) and thus is an element of $\nmax(P')$. We show that $j$ is also an element of $\nmax(P)$. The only possible obstructions to this are $-e_i +e_j$ or $e_i+e_j$ being in $P$.  We know its impossible for $e_i+e_j$ to be in $P$ since its opposite is.  If $-e_i+e_j$ were to be in $P$, then by transitivity $-e_i$ would be in $P$ which we assumed was not the case much earlier. So $j$ is an element of $\nmax(P)$ and thus $i$ satisfies (iv).
    \item Suppose $j$ satisfies (vi) and thus $e_j$ is in $P'$ and thus $P$.  However, by transitivity with $-e_i - e_j$ this implies that $-e_i$ is in $P$, which we earlier assumed was not the case.  
    \item Suppose $j$ satisfies (vii) and thus $e_j + e_p$ is in $P$, where $p \in \pmax(P')$.  By transitivity, we see that $-e_i +e_p\in P$. It suffices to show that $p$ is in $\pmax(P)$.  Suppose not.  Then either $e_i -e_p$ or $-e_i-e_p$ is in $P$. We see that $e_i-e_p$ cannot be in $\pmax(P)$ since its opposite $-e_i +e_p$ was already shown to be in $P$. If $-e_i-e_p$ is in P, then by transitivity with $-e_i+e_p$ we get that $-e_i$ is in $P$, which we assumed much earlier not to be the case. Thus, $p$ is in $\pmax(P)$ and thus $i$ satisfies (iii).
    \item Suppose $j$ satisfies (viii) and $e_j - e_n$ is in $P$, where $n$ is in $\nmax(P')$.  By transitivity, we know that $-e_i-e_n$ is in $P$.  To show $i$ satisfies (iv), it suffices to show that $n$ is in $\nmax(P)$.  The only obstructions to this are $-e_i+e_n$ or $e_i+e_n$ being in $P$.  We see that $e_i+e_n$ cannot be in $P$ because we already established that its opposite is.  Suppose that $-e_i+e_n$ is in $P$.  Then by transitivity with $-e_i-e_n$, we know that  $-e_i$ is in $P$, which we already assumed was not the case.  Thus, $n$ must be in $\nmax{P}$ and so $i$ satisfies (iv).  
    
\end{itemize}

Thus, in all cases, (i) must satisfy at least one of (i), (ii), (iii), or (iv). The proof that the $i$ must satisfy at least one of (v), (vi), (vii), or (viii) is similar.  

We have thus shown that the given inequalities in this proposition do describe $\op{P}$.  We now argue that this set of inequalities is a minimal set describing $\op{P}$. From the definition of $\minrep(P)$, we see that we cannot remove any of the hyperplanes of the form $\langle\alpha, x\rangle\geq 0$ without changing $\op{P}$.  Now let $i \in \pmax(P)$.  All of the adjacent relations to $i$ must be of the form $e_i-e_j$, $e_i+e_j$, or $e_i$. So all of restrictions on $x_i$ coming from rest of the poset are of the form $x_i\geq x_j$, $x_i \geq -x_j$ and $x_i \geq 0$.  No combination of these can imply $x_i\leq 1$, so this inequality is necessary in this description of~$\op{P}$.  
\end{proof}

 In order to describe a convex hull description of $\op{P}$, we first need the idea of filters. For classical posets, they are defined as follows, as appears in, e.g., \cite{B18}.

\begin{defin}
Let $\Pi$ be a classical poset.  Then, a filter is a subset $F$ of $\Pi$ such that for $a\leq_\Pi b $, $a\in F$ implies $b\in F$.
\end{defin}

The filters give a vertex description of $K_\Pi$ through the following
proposition, which again appeared in \cite{S86}.

\begin{prop}[Stanley~\cite{S86}]
The vertex set of $\mathcal{O}(\Pi)$ is the set of all $e_F$, where $F$ is a filter of $\Pi$ and $e_F: \Pi \to \{0,1\}$ is the point in $\R^\Pi$ given by 
\[e_F(a) \ = \ \left\{\begin{array}{ll}
       1 & \text{if } a \in F,\\
      0 & \text{otherwise.}
        \end{array}\right. \]
\end{prop}

Note that the filters are exactly the points $x \in \{0,1\}^\Pi$ such that
$\langle\alpha, x\rangle\geq 0$ for all $\alpha \in f(\Pi)$.  (Recall that $f$
is the map from a classical poset to its corresponding subset of the type-$A$ root system).  This observation allows for a similar definition for signed posets.  The following definition appears in \cite{R93}.

\begin{defin}
Let $P$ be a signed poset on $[n]$.  Then $x \in \{-1,0,1\}^n$ is a \emph{signed filter} of $P$ if $\langle\alpha, x\rangle\geq 0$ for all $\alpha \in P$.  
\end{defin}

Note that in \cite{R93}, due to the earlier mentioned difference in convention
(up to a sign) in defining the map from classical posets to subsets of the
type-$A$ root system, Reiner calls these objects signed order ideals instead of signed filters.  

We can now give a convex hull description of $\op{P}$.  

\begin{prop}\label{prop:signed filters}
Let $P$ be a signed poset on $[n]$ and let $F$ be the set of signed filters on $P$.  Then $\op{P} = \conv(F)$.
\end{prop}

\begin{rem}
	A signed filter may not necessarily be a vertex of $\op{P}$ (in contrast to classical
order polytopes).  For example $x = (0,\dots, 0)$ is always a signed filter of $\op{P}$,
but need not be a vertex (as seen in Figure \ref{fig:examples}).

\end{rem}

Before we give the proof of this proposition, we need to discuss a specific unimodular
triangulation of $\op{P}$ for any signed poset.  Along the way, we discuss a way of dividing $\R^n$ into cones indexed by elements of the signed permutation group $S_n^B$.

\begin{defin}\label{def:permcones}
Consider a signed permutation $\sigma = (\pi, \epsilon) \in S_n^B$.  The simplicial cone associated with $\sigma$ is
\[K_\sigma \ := \ \left\{x \in \R^n: \, 0\leq \epsilon_1x_{\pi_1} \leq \dots \leq
\epsilon_nx_{\pi_n} \right\}.\]
\end{defin}

These simplicial cones induce a triangulation on the $[-1,1]^n$ cube, as described in 
\cite[Proof of Theorem 6.9]{bb13}.  Each maximal face of this triangulation is of the form $\Delta_\sigma := \{x \in \R_n: 0\leq \epsilon_1x_{\pi_1} \leq \dots \leq \epsilon_nx_{\pi_n}\leq 1\}$.
Note that the defining hyperplanes of $\op{P}$ are a subset of the union of the set of hyperplanes defining the described triangulation of $[-1,1]^n$.  Thus, restricting this triangulation to $\op{P}$ gives a triangulation of $\op{P}$.  We will refer to this triangulation as $\mathcal{T}$.  Later, we will use the fact the maximal faces $\mathcal{T}$ are exactly the set $\Delta_\sigma$, where $\sigma \in \JH(P)$.

Now, we will use this triangulation to prove Proposition \ref{prop:signed filters}.

\begin{proof}
First, by definition, all the points in $F$ are in $\op{P}$, and thus the convex hull of $F$ is in~$\op{P}$. 

We now show that $\op{P} \subseteq \conv(F)$.  Suppose we have some point $x \in \op{P}$.
Then, $x$ must be in one of the full-dimensional simplices of $\mathcal{T}$, i.e., $x \in \Delta_\omega$ for some $\omega \in S_n^B$, and thus can be written as a convex combination of the vertices of $\Delta_\omega$.  These vertices have entries of $0,1,-1$, and thus are filters of~$P$.
\end{proof}

%
%
%
%
%
%
%
%
%
%
%
%
%


\section{Computing the $h^\ast$-polynomial of $\op{P}$}\label{sec:hstar}

In 1962, for a lattice polytope $P\subset \mathbb{Z}^d$, Ehrhart \cite{E67} introduced and
proved polynomiality of the counting function $\ehr_P(t)=|tP\cap\mathbb{Z^d}|$. Here the \emph{$t^{th}$ dilate of $P$} is $tP:=\{tp\in\R^n : p\in P\}$.
 \begin{thm}(Ehrhart's Theorem)\label{thm: ehrhart}
  For any $d$-dimensional lattice polytope $P\subseteq \R^n$, the quantity $\ehr_P(t)=|tP\cap\mathbb{Z}^n|$ agrees with a polynomial of degree $d$.  
 \end{thm}
 This \emph{Ehrhart polynomial} can be seen as a discrete measure of volume.  In general, the coefficients of the Ehrhart polynomial of a lattice polytope are rational numbers that can be negative. The generating series $\mathrm{Ehr}(P;z)$ of a $d$-dimensional lattice polytope
 can be written as 
\[\mathrm{Ehr}(P;z) = 1 + \sum_{t\geq 1} \ehr_P(t)z^t = \frac{h^*_0 + h^*_1 z + \dots +
h^*_{d+1}z^{d}}{(1-z)^{d+1}} \, . \]
The numerator of this expression is called the \emph{$h^*$-polynomial} of $P$ and denoted $h^*(P;z)$.  It can be helpful to view the information encoded by the Ehrhart polynomial in this form. For example, unlike the Ehrhart polynomial,  the coefficients of the $h^*$-polynomial of a lattice polytope are always non-negative integers \cite{stanley1980decompositions}. Thus, one area of research in Ehrhart theory is to give combinatorial interpretations of these coefficients for specific families of polytopes.
For classical order polytopes,  we can describe the $h^*$-polynomial in terms of
permutation statistics in the following way, as outlined, e.g, in
\cite[Chapter~5]{B18}.  The
Jordan--H\"{o}lder set of a classical poset is defined similarly as for a
signed poset.  

\begin{defin}\label{def: classic jh}
Let $\Pi$ be a classical poset on $[n]$.  The Jordan-H\"{o}lder set of $\Pi$ is the set of permutations $\omega \in S_n$ that are order-preserving, that is, $\langle\alpha,\omega \rangle\geq 0$ for all $\alpha \in f(\Pi)$, where $f$ is the map described in Proposition \ref{prop:coxposet} and where we think of  $\omega$ as the point $(\omega_1, \dots, \omega_n) \in \R^n$.  
\end{defin}

\begin{defin}
Let $\omega = \omega_1\omega_2\dots\omega_n \in S_n$. A \emph{descent} is a position $i$ such that $\omega_i>\omega_{i+1}$.  The \emph{descent set} of $\omega$ is $\mbox{Des}(\omega):=\{i:i \mbox{ is a descent of } \omega\}$, and the \emph{descent statistic} of $\omega$ is $\mbox{des}(\omega):=|\mbox{Des}(\omega)|$.
\end{defin}

See, for example,  \cite[Section~3.12]{stanleyec1} for a more complete discussion of permutation statistics and other equivalent definitions of the classical Jordan--H\"{o}lder set, relating for linear extensions of posets. 

\begin{prop}\label{prop: classic h star}
	Let $\Pi$ be a naturally labeled poset with Jordan--H\"{o}lder set $\JH(\Pi) \subset S_n$.  Then
\[h^*_{O_\Pi}(z) \ = \sum_{\tau \in JH(\Pi)}z^{\mathrm{des}(\tau)}.\]
\end{prop}

In this section, we use the triangulation  $\mathcal{T}$ described in
Section~\ref{sec:altdes} to give an analogous description of the
$h^\ast$-polynomial of $\op{P}$ in terms of statistics on the type-$B$
permutation group $S_n^B$. 
We first introduce some background on half-open polytopes and half-open decompositions of polytopes.  A more complete treatment of this material can be found in \cite[ Chapter 5]{B18}.
We first define what it means for a point to be beyond a face $F$ of a polytope
$P$ in the special case in which $F$ is a facet of $P$.  For more generality,
see \cite[Chapter~3]{B18}. 

\begin{defin}
Let $P \subset \R^n$ be a full-dimensional polytope, and let $F$ be a facet of $P$
with defining hyperplane $\langle a,x\rangle = b$ such that $P$ lies in the half space $\langle a,x\rangle \leq b$.  Then $p\in \R^n$ is  \emph{beyond} F if $\langle a,p \rangle > b$.
\end{defin} 

This concept of a point being beyond a facet is used to construct half-open polytopes.  

\begin{defin}
Let $P \subset \R^n$ be a full-dimensional polytope with facets $F_1, \dots F_m$.  Let $q
\in \R^n$ be generic relative to $P$, i.e., $q$ does not lie on any facet-defining hyperplane of $P$.  Then we define
\[\HH_qP \ := \ P \setminus \bigcup_{i \in I} F_i \, , \]
where $I := \{i \in [m]: q \mbox{ beyond } F_i \}$.  We call $\HH_qP$ a
\emph{half-open polytope}. 
\end{defin}

Applying this construction to a triangulation of $P$ allows us to write $P$ as a disjoint union of half-open simplices.
 
\begin{lem}[see, e.g.,~\cite{B18}]\label{lem: disjoint union}
Let $P\subset \R^n$ be a full-dimensional polytope with dissection $P = P_1
\cup P_2 \cup \dots \cup P_m$.  If $q \in P^\circ$ is generic relative to each $P_j$, then 
\[P \ = \ \HH_qP_1\uplus \HH_qP_2 \uplus \dots \uplus \HH_qP_m \, . \]
\end{lem}

We can apply these results to the unimodular triangulation  $\mathcal{T}$ of
$\op{P}$  described in the previous section in order to write $\op{P}$ as a
disjoint union of half-open unimodular simplices.  These simplices can be
described in terms of the naturally ordered descent statistic of $S_n^B$.  For
more information about various statistics (including several definitions of the
descent set) of $B_n$, see, e.g.,~\cite{bb13}.

\begin{defin}
For $\sigma \in S_n^B$, the \emph{naturally ordered descent set} of $\sigma$ is 
\[\mbox{NatDes}(\sigma) \ := \ \left\{i \in \{0, \dots, n-1\}: \, \sigma(i)>\sigma(i+1)
\right\} , \]
where we use the convention that $\sigma(0) = 0$.   The \emph{natural descent
statistic} of $\sigma$  is $\mbox{natdes}(\sigma) := |\mbox{NatDes}(\sigma)|$.
\end{defin}

\begin{prop}  \label{prop: half open simplex}
Let $ \sigma = (\pi, \epsilon) \in S_n^B$ and $p := \left(\frac{1}{n+1},
\frac{2}{n+1}, \dots, \frac{n}{n+1}\right)$. Then
\[\mathbb{H}_p\Delta_\sigma \ = \ \left\{x \in \R^\Pi: \begin{array}{ll}
        0\leq \epsilon_1x_{\pi_1} \leq \dots \leq \epsilon_nx_{\pi_n}\leq 1\\
       \epsilon_ix_{\pi_1} < \epsilon_{i+1}x_{\pi_{i+1}} & \text{if } i \in
\mbox{\rm NatDes}(\sigma)\\
      0< \epsilon_1x_{\pi_1} & \text{if } 0 \in \mbox{\rm NatDes}(\sigma)
        \end{array}\right\}. \]
\end{prop} 

\begin{proof}
We identify the facets of $\Delta_\sigma$ that are removed in the half-open
polytope $\mathbb{H}_p\Delta_\sigma$, starting with the facets of the form
$\epsilon_ix_{\pi_i} = \epsilon_{i+1}x_{\pi_{x+1}}$ for $i \in [n-1]$.  These
facets are removed when $p$ is beyond the facet, which occurs exactly when
$\epsilon_ip_{\pi_i} > \epsilon_{i+1}p_{\pi_{i+1}}$.  Substituting our
expressions for the coordinates of $p$ yields
\begin{equation*}
\epsilon_i\frac{\pi_i}{n+1} \ > \ \epsilon_{i+1}\frac{\pi_{i+1}}{n+1} \, ,
\end{equation*}
which simplifies to $\sigma(i)>\sigma(i+1)$. Thus, we see that a facet of $\Delta_\sigma$ of the form $\epsilon_ix_{\pi_i} = \epsilon_{i+1}x_{\pi_{i+1}}$  is removed exactly when $i \in \mbox{NatDes}(\sigma)$.  

We now consider the facet given by $\epsilon_1x_{\pi_1}=0$.  We know that $p$ is beyond this facet exactly when $\epsilon_1p_{\pi_1}<0$.  Since all the coordinates of $p$ are positive, this holds exactly when $\epsilon_1<0$, which coincides with the cases in which $0 \in \mbox{NatDes}(\sigma)$. 

We finally consider the facet given by $\epsilon_nx_{\pi_n}= 1$.  Since $-1\leq p_i \leq 1$ for any $i \in [n]$, we know that $p$ is never beyond this facet.  Thus, this facet is never removed in $\mathbb{H}_p\Delta_\sigma$.  
\end{proof}

A classical result (see, e.g., \cite{B18}) gives the $h^*$-polynomials of half-open unimodular simplices.  

\begin{lem}\label{lem: h star simp}
Let $\mathbb{H}_p\Delta$ be a unimodular half open simplex with $k$ missing facets.  Then, $h^*_{\mathbb{H}_p\Delta}(z) =z^k$.  
\end{lem}

We can now describe the $h^*$-polynomial of $\op{P}$ for any naturally ordered signed poset $P$ in terms of descent statistics in a statement analogous to Proposition \ref{prop: classic h star}.  

\begin{prop}\label{prop: signed h star}
Let $P$ be a naturally labeled signed poset on $[n]$ with  Jordan--H\"{o}lder set $\JH(P) \subset B_n$.  Then
\[h^*_{\op{P}}(z) \ = \sum_{\tau \in \JH(P)}z^{\natdes(\tau)}.\]
\end{prop}
\begin{proof}
Since $P$ is naturally labeled, we know that $p =
\left(\frac{1}{n+1},\dots,\frac{n}{n+1}\right)$ is an interior point of
$\op{P}$.  Thus, we can use the triangulation $\mathcal{T}$ restricted to
$\op{P}$  to decompose $\op{P}$ into a disjoint union of half-open simplices.
Using Lemma \ref{lem: disjoint union} with respect to the point $p$ and this
triangulation, we obtain
	\[\op{P} \ = \biguplus_{\sigma \in \JH(P)} \mathbb{H}_p\Delta_\sigma.\]

From Proposition \ref{prop: half open simplex}, we know that
$\mathbb{H}_p\Delta_\sigma$ is a unimodular half-open simplex with
$\natdes(\sigma)$ missing facets.  From Lemma \ref{lem: h star simp}, we know that the $h^*$- polynomial of such an object  is $z^{\natdes(\sigma)}$.  Since the $h^*$-polynomials of disjoint half-open polytopes are additive, 
\[h^*_{\op{P}}(z) \ = \sum_{\tau \in \JH(P)}z^{\natdes(\tau)}. \qedhere \]
\end{proof}

\begin{rem}
This result gives a description of the $h^*$-polynomial only for naturally labeled signed posets.  However, since unimodularly equivalent polytopes have identical $h^*$-polynomials, this encompasses all the unique $h^*$-polynomials corresponding to signed posets by Propsitions \ref{prop: geo iso nice} and~\ref{prop: every poset nat label}.
\end{rem}


\section{Which Signed Order Polytopes Are Gorenstein?}\label{sec:gorenstein} 

We now review a classification of Gorenstein order polytopes in the classical case and discuss how it extends to signed order polytopes.

\begin{defin}
A lattice polytope is \emph{Gorenstein} if there exists a positive integer $k$
such that $(k-1)P^\circ \cap \Z^d = \emptyset$, $|kP^\circ \cap \Z^d|=1$, and
$|tP^\circ \cap \Z^d |=|(t-k)P^\circ \cap \Z^d|$ for all integers $t > k$. 
\end{defin}

This is equivalent to the polytope having a symmetric $h^*$-vector. For
classical posets, the following result is well known (see, e.g., \cite{B18}):

\begin{prop}\label{classical gorenstein}
The order polytope of a poset $P$ is Gorenstein  if and only if $P$ is graded
(i.e., all maximal chains have the same length).
\end{prop} 

In this section, we will develop an analogue of this result for signed posets; we begin with a representation of a signed poset on $[n]$ as a classical poset on $[2n+1]$ that satisfies certain properties, first introduced in~\cite{fischer93}.

\begin{defin}
Let $P$ be a signed poset on $[n]$.  The \emph{Fischer represention} $\hat{G}(P)$ is a poset on $[-n,n] = \{-n, -(n-1),\dots, -1, 0, 1,\dots, n-1, n\}$ whose relations are the transitive closure of the following:
\begin{align*} 
i<j &\mbox{ and } -j < -i  & &\mbox{ for } -e_i + e_j \in P \\ 
i<-j &\mbox{ and } j < -i & &\mbox{ for } -e_i - e_j \in P \\ 
-i<j &\mbox{ and } -j < i & &\mbox{ for } e_i + e_j \in P \\ 
i<0 &\mbox{ and } 0 < -i & &\mbox{ for } -e_i \in P \\ 
-i < 0&\mbox{ and } 0< i & &\mbox{ for } e_i \in P \, .
\end{align*}
\end{defin}

Figure \ref{fig: fischer} shows an example of the bidirected graph representation and the Fischer representation of the same signed poset.

\begin{figure}[ht]
    \centering
    \includegraphics[width = .8\textwidth]{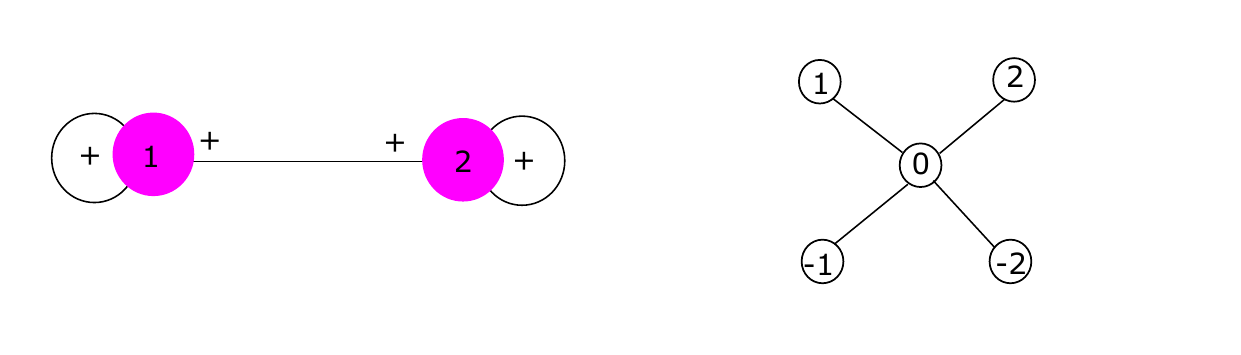}
    \caption{The left side shows the bidirected graph representation of $P := \{e_1, e_2, e_1+e_2\}$ and the right side shows the Fischer representation of~$P$. }
    \label{fig: fischer}
\end{figure}

\begin{rem}
This definition has all the inequalities reversed from Fischer's original definition to make this poset consistent with our definition of signed order polytopes.
\end{rem}

The following proposition from \cite{fischer93} classifies exactly when a poset  on
$[-n,n]$ equals $\hat{G}(P)$ for some signed poset $P$ on~$[n]$. 

\begin{prop}\label{fischer symmetry}
 A poset on $[-n,n]$ is $\hat{G}(P)$ for some signed poset $P$ if and only if 
\begin{itemize}
\item $i<j$ if and only if $-j<-i$ for all $i,j \in [-n,n] $;
\item if $-i<i$ then $-i < 0<i$ for all $i \in [-n,n]$.  
\end{itemize}

\end{prop}

Next, we establish how we can view the previously defined signed order polytopes through this lens.  

\begin{prop}\label{gorensteinhelper}
Let P be a signed poset on $[n]$.  Define a polytope $\op{{\hat{G}(P)}} \subset \R^n$ via the following inequalities:

\begin{itemize}
\item  $-1 \leq x_i \leq 1$ for all $i \in [n]$;
\item $-x_i \leq x_j$ for all $-i \leq j$, where $i, j \in [n]$;
\item $x_i \leq -x_ j$ for all $i \leq -j$, where $i, j \in [n]$;
\item $x_i \leq x_j$ for all $i \leq j$ where $i, j \in [n]$;
\item $x_i \geq 0$ for all $i \geq 0$ where $i \in [n]$;
\item $x_i \leq 0$ for all $i \leq 0$ where $i \in [n]$.  
\end{itemize}
(Note that some of these inequalities will be equivalent to each other.) Then $\op{P} = \op{{\hat{G}(P)}}$.
\end{prop}

\begin{proof}
Starting with a signed poset $P$, constructing $\op{P}$ and $ \op{{\hat{G}(P)}}$
yields polytopes defined by exactly the same set of inequalities.  
\end{proof}

We can now give the following analogue to Proposition \ref{classical gorenstein}.

\begin{prop}\label{prop: gor result}
Let $P$ be a signed poset on $[n]$.  The signed order polytope $\op{P}$ is Gorenstein if and only if $\hat{G}(P)$ is graded.  
\end{prop}


\begin{proof}
Suppose $\hat{G}(P)$ is graded. We first briefly establish some facts about its
maximal chains.  The element $0$ must be in one of the maximal chains; consider
the part of the chain $0 < c_1 \dots < c_k$.  Because of Proposition
\ref{fischer symmetry}, the chain $-c_k< \dots<-c_1<0$ must also exist in
$\hat{G}(P)$, and when extended into a maximal chain must not contain any other
elements below $0$, otherwise $0 < c_1 \dots < c_k$ could be similarly extended.
Thus, $-c_k < -c_{k-1} < \dots -c_1 < 0<c_1<\dots<c_{k-1}<c_k$ is a maximal
chain.  Thus, all maximal chains in $\hat{G}(P)$ are of the same \emph{even} length.

Suppose $\hat{G}(P)$ is graded with  maximal chains of length $2k-2$ and rank function $\rho: \hat{G}(P) \rightarrow \mathbb{N}$. We show
that $\op{P}$ is Gorenstein of degree $k$.  We first verify that $(k-1) \op{P}$
has no interior points.  Consider a maximal chain $c_{-(k-1)} < c_{-(k-2)}
< \dots < c_0 < \dots < c_{k-2} < c_{k-1}$ of $\hat{G}(P)$.  In
order for a point $q$ to be in the interior of $(k-1)\op{p}$, by Proposition \ref{gorensteinhelper}, $q$ must satisfy  $-(k-1)<q_{c_{-(k-1)}} < q_{c_{-(k-2)}} < \dots < q_{c_0} < \dots < q_{c_{k-2}} < q_{c_{k-1}}<k-1$. This is not possible, since there are not $2k-1$ distinct integers between $-(k-1)$ and~$(k-1)$.  


We now construct a point $p \in k\op{P}$ and show that it is the unique interior
point of $k\op{P}$.  If in $\hat{G}(P) $, $\rho(i) = \rho(0)$ for some $i \in [n]$,
we note that  because of the symmetries outlined in Proposition \ref{fischer symmetry}, it must also be true that $\rho(-i)= \rho(0)$. In this case, we set $p_i = 0$.  Suppose $\rho(i) - \rho(0) = \ell$.  Then, we assign   $p_i = \ell$.  Note that $-(k-1)\leq \ell \leq k-1$, so $p$ satisfies the strict inequality $-(k-1)<x_i<k-1$ for all $i \in [n]$.  
By construction, the coordinates of $p$ satisfy the other strict inequalities in $k\op{{\hat{G}(P)}}$. Thus $p$ is an interior point of~$k \op{P}$.

We now show that such an interior point must be unique. As described above,
since $p$ is an interior point of $k \op{P}$, for every maximal chain in
$\hat{G}(P)$,  $c_{-(k-1)} < c_{-(k-2)} < \dots < c_0 < \dots < c_{k-2} < c_{k-1}$, $p$ must satisfy  $-k<q_{c_{-(k-1)}} < q_{c_{-(k-2)}} < \dots < q_{c_0} < \dots < q_{c_{k-2}} < q_{c_{k-1}}<k$, so each coordinate corresponding to this maximal chain is uniquely determined.  Since every element in a poset is part of a maximal chain, every coordinate of $p$ is uniquely determined. 

Finally, for all integers $t \geq k$, we establish a bijection between the sets $t\op{P}^\circ \cap \Z^d$ and $(t-k)P\cap \Z^d$. 
 Let $p \in  (t-k)P\cap \Z^d$ and consider $\phi(p) = p + (\rho(1) - \rho(0), \dots, \rho(n) - \rho(0)$).
  First, we know that $t<\phi(p)_i<t$, since $t-k \leq p_i \leq t-k$ and $ -k<\rho(i) - \rho(0)<k$. 
 Furthermore, since $p$ satisfies the inequalities of  $k\hat{G}(P)$ and $(\rho(1) - \rho(0), \dots, \rho(n) - \rho(0))$ satisfy the strict inequalities, their sum $\phi(p)$ also satisfies the strict inequalities, so $\phi(p)$ is indeed an interior point of $t\op{P}$.  
We now need to show that $\phi: (t-k)\op{P} \rightarrow t \op{P}^\circ$ is bijective.
  It suffices to show that the map is surjective.
  Suppose we have a point $q \in t\op{P}^\circ$, and consider $q' = q - (\rho(1) - \rho(0), \dots, \rho(n) - \rho(0))$, so that $\phi(q') = q$.  Note that since $-(t-1)\leq q_i \leq t-1$ and $-(k-1)\leq\rho(i)-\rho(0)\leq k-1$ , $-(t-k) \leq q'_i \leq t-k$.
  Now, suppose $i \leq j \in \hat{G}(P)$.  Then, since $q$ is in the relative interior of $t\op{P}$, $q_i  - \rho(i) \leq q_j  - \rho(j)$, which implies $q'_i \leq q'_j$.  So $q'$ is indeed in $(t-k)\op{P}$.

We now suppose that $\hat{G}(P)$ is not graded, and show that $\op{P}$ cannot be Gorenstein. For $\op{P}$ to be Gorenstein of degree $k$, the interior points of the cone $\hom{\op{P}}$ must be the integer points of the shifted cone $\hom{\op{P}} + (p, k)$, where $p$ is the unique interior point of $k\op{P}$.  Let $2a-2$ be the length of the longest maximal chain in  $\op{{\hat{G}(P)}}$.
 Based on an argument above, we know that there are no interior points in $b\op{P}$ for any nonnegative integer $b<a$. 
 We now consider $a\op{P}$, the first dilate with at least one interior point $p$.   
However, since $\hat{G}(P)$ is not graded, there must be a maximal chain of a length $k$ where $k<2a-2$, $  c_{1}  < \dots c_{k}$.
   Without loss of generality, let $p_{c_{1}} = -(a-1)$ and $p_{c_k} = a-1$, which is always possible since there are no elements above $c_k$ nor below $c_{1}$. 
 Since the maximal chain we are considering has length less than $2a-a$, we know that there is some   $1\leq i \leq k-1$ such that $p_{c_{i+1}} - p_{c_i} > 1$.
  We now construct an interior point $p'$ of $(a+1) \op{P}$, where we add $1$ to all the coordinates $p_j$ where $j \geq c_i$ in $\hat{G}(P)$ except for $p_{c_{i+1}}.$ 
Note that $(p',a+1) \in \hom{\op{P}}$ is formed from adding a point that is not compatible with the ordering to $(p,a)$, so  $(p',a+1)$ is an interior point of $\hom{\op{P}}$ that does not lie in the shifted version. 
 Thus, $\op{P}$ is not Gorenstein.  
\end{proof}

The following result applies to our situation.  

\begin{thm}[Bruns--R{\"o}mer~\cite{BH07}]
 A Gorenstein lattice polytope $P$ with a regular unimodular triangulation has a unimodal $h^*$-vector.
\end{thm}

\begin{cor}
Let $P$ be a signed poset on $[n]$.  If $\hat{G}(P)$ is graded, then the $h^*$-polynomial of $P$ is unimodal.
\end{cor}

Stembridge~\cite{S08} extended Reiner's work with signed posets to any root system. He defines a generalization of order cones and signed order cones for other root systems, calling these Coxeter cones.
\begin{defin}[Stembridge~\cite{S08}]
Let $\Phi$ be any root system in $\mathbb{R}^n$, and let $\Psi$ be a subset of
$\Phi$.  Then the \emph{Coxeter cone} of $\Psi$ is
\[\Delta(\Psi) \ := \ \left\{x \in \mathbb{R}^n: \langle x, \beta \rangle \geq 0 \mbox{ for any } \beta \in \Psi \right\} . \]
\end{defin}
Viewing these cones as simplicial complexes and defining a general notion of when these complexes are graded, he used algebraic methods to give a condition on when the $h$-vectors of these cones are symmetric and unimodal.  

\begin{thm}[Stembridge~\cite{S08}]  \label{coxcones} If a Coxeter cone is \df{graded}, then its $h$-polynomial is symmetric and unimodal.
\end{thm}

The definition of \emph{graded} is quite technical; for a full definition see \cite{S08}.
 
Proposition \ref{prop: gor result} can be seen as an Ehrhart-theoretic
interpretation of the type-$B$ case of Theorem~\ref{coxcones}, using a geometric proof method as opposed to the algebraic proof method in \cite{S08}.  Below, we summarize the connection between Proposition \ref{prop: gor result} and Theorem~\ref{coxcones}. 

We first interpret Theorem~\ref{coxcones} in the type-$B$ case.  In Examples 5.2(b) and
6.4(b), Stembridge notes that his definition of a graded Coxeter cone, when restricted to
the type-$B$ case, results in exactly the Coxeter cones of type $B$ corresponding to signed posets with a graded Fischer representation.  Thus, these Coxeter cones result in simplicial complexes with a symmetric and unimodal $h$-polynomial.  

We then make the transition from the $h$-polynomial of a type-$B$ Coxeter cone to the $h^*$-polynomial of $\op{P}$ of the corresponding signed poset $P$.  In Section 4 of \cite{S08}, Stembridge notes that the $h$-polynomial of his type $A$ and $B$ Coxeter cones are identical to the $h^*$-polynomial of a certain lattice polytope.  His construction in the type-$B$ case, described in algebraic terms, gives the same polytope as $\op{P}$. Thus we see that Proposition \ref{prop: gor result} can be seen as a special case of the broad algebraic result in \cite{S08}.


\section{Chain Polytopes}\label{sec:chainpoly}

In \cite{S86}, Stanley establishes some properties of chain polytopes.

\begin{defin}
Let $\Pi$ be a poset.  An \df{antichain} of $\Pi$ is a subset $I$ of the elements of $\Pi$ such that for any $i,j \in I$, neither $i <j$ nor $j<i$ in $\Pi$.  
\end{defin}

\begin{prop}[Stanley~\cite{S86}]\label{prop: antichain}
The vertices of $\mathcal{C}(\Pi)$ are given by the $\{0,1\}$-indicator vectors of the antichains of $\Pi$.
\end{prop}

\begin{thm}[Stanley~\cite{S86}] \label{thm: order and chain} Let $\Pi$ be a poset on $[n]$. 
\begin{itemize}
\item $\mathcal{C}(\Pi)$ and $\mathcal{O}(\Pi)$  have the same $h^\ast$-polynomial.
\item $\mathcal{C}(\Pi)$ and $\mathcal{O}(\Pi)$  are combinatorially equivalent if and only if $\Pi$ does not contain the poset shown in Figure \ref{fig: forbidden}  as a subposet.
\end{itemize}
\end{thm}


\begin{figure}[ht]
    \centering
    \includegraphics[width=.2\textwidth]{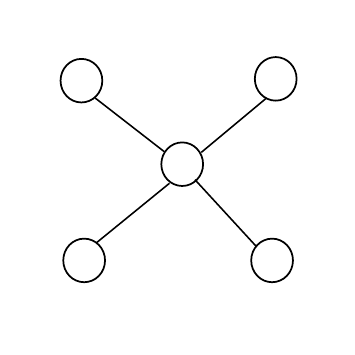}
    \caption{The forbidden poset in Theorem \ref{thm: order and chain}. }

    \label{fig: forbidden}
\end{figure}

In this section, we suggest a definition for signed chain polytopes and examine the properties of said polytopes. 
(We do not claim this is the only sensible definition for the concept of a
signed chain polytope.)
First, we define a chain of a signed poset.  

\begin{defin}
A chain on a signed poset $P$ on $[n]$ is an ordered pair $(C,S)$, where $C = (c_1, \dots, c_m) \in [n]^m$ and $S = (s_1, \dots, s_{m-1}) \in \{-1,1\}^{m-1}$ such that for each $i \in [m-1]$ there exists $\alpha_i \in P$ that satisfy:

\begin{itemize}
\item if $s_i = 1$, then $\alpha_i = \pm(e_{c_i} - e_{c_{i+1}})$, and if $s_i = -1$, then $\alpha_i = \pm(e_{c_i} + e_{c_{i+1}})$;
\item $\alpha_i +\alpha_{i+1} \in P$.
\end{itemize}
\end{defin}

We now give one definition of a chain polytope.

\begin{defin} \label{def: signed chain}
The \df{signed chain polytope} $\mathcal{C}_P$ of a signed poset $P$ on $[n]$ is the intersection of inequalities of the form
\[-1\leq x_{c_1} + s_1x_{c_2} + s_1s_2x_{c_3} + \dots + s_1s_2\dots s_{m-1}x_{c_m} \leq 1\]
for each chain $(C,S)$ of $P$. 
\end{defin}

We note that, in the definition of the chain polytope for a classical poset, it
suffices to have an inequality for each \emph{maximal} chain, since the other
inequalities are implied by these. The equivalent statement is not true for
chain polytopes of signed posets. 

We also introduce another useful class of polytopes, directly related to Gorenstein polytopes. 
\begin{defin}
A lattice polytope is \df{reflexive} if its hyperplane description can be written as $Ax \leq 1$ for an integral matrix~$A$. 
\end{defin}
\noindent There are many equivalent definitions of reflexive polytopes, for example relating to the \df{duals} of polytopes.  A reflexive polytope can also be described as a Gorenstein polytope with Gorenstein index $1$.  For a full description, see, for example, \cite{Z12}. 

Directly from the definitions above, we make the following observation:

\begin{prop}
For any signed poset $P$, the signed chain polytope $\mathcal{C}_P$ is reflexive, and thus Gorenstein.
\end{prop}

\begin{proof}
From the definition, we can see that $\mathcal{C}_P$ is defined by a linear
system $Ax \leq 1$ for an integral matrix~$A$. 
\end{proof}

One consequence of this is that it allows us to associate a Gorenstein polytope with every classical poset, since every classical poset can be viewed as a signed poset.

Similarly to Proposition \ref{prop: antichain}, we can give a convex hull description of $\mathcal{C}_P$ in terms of antichains of $P$, defined below.

\begin{defin}
Let $P$ be a signed poset on $[n]$.  An element of $a = (a_1 \dots a_n) \in  \{-1,0,1\}^n$ is an antichain  of $P$ if for each element $\alpha \in P$ of the form $\pm e_i \pm e_j$ or $\pm e_i \mp e_j$, $\langle \alpha,   a \rangle \neq 0$ unless $a_i = a_j = 0$.
\end{defin}
\begin{example} Figure \ref{fig: chain ex} shows a signed poset, with a chain indicated in blue.  This chain $(C,S) = ((1,2,3),(1,-1))$ is the longest chain in this signed poset.

There are many antichains of this signed poset, one of which is $a =(1,0,1,-1)$. It might seem like since $1$ and $3$ are related, they shouldn't both have a nonzero entry in the antichain, but the way the signs are arranged makes $a$ fit the definition. 
\end{example}
\begin{figure}[ht]
	\centering
        \includegraphics[width=.3\textwidth]{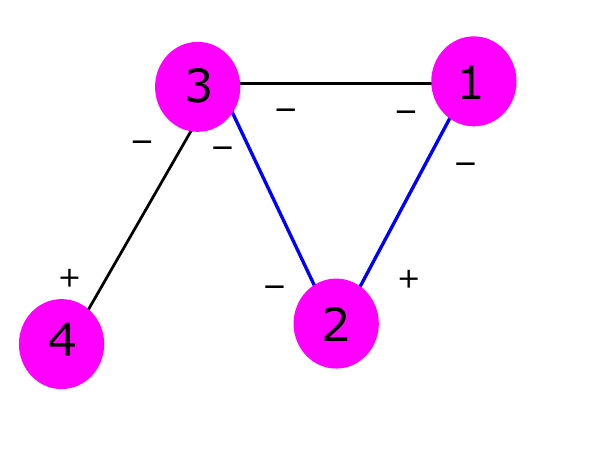}
        \caption{A bidirected graph representation of a signed poset on $4$ elements, with a chain highlighted in blue.  }
        \label{fig: chain ex}
        \end{figure}
\begin{prop}
The set of antichains of a signed poset $P$ are exactly the integer points in $\mathcal{C}_P$.  
\end{prop}

\begin{proof}

Suppose a point $p \in \R^n$ is an integer point of $\mathcal{C}(P)$.  Then $p$ must satisfy all the inequalities specified in Definition \ref{def: signed chain}; in particular:
\begin{itemize}
\item Suppose $\alpha \in P$ is of the form $\pm e_i \pm e_j$.  Then, $(C,S)= (\{i,j\},\{-1\})$ is a signed chain of $P$, and $p$ must satisfy the inequalities $-1\leq p_i - p_j \leq 1$.  Thus, either $p_i = p_j = 0$, $p_i = \pm 1$ and $p_j = 0$, $p_i= 0$  and $p_j = \pm 1$, or $p_i =p_j = \pm 1$.  In the latter three cases, it is true that $\langle p, \alpha\rangle \neq 0$.  
\item Suppose  $\alpha \in P$ is of the form $\pm e_i \mp e_j$.  Then, $(C,S)= (\{i,j\},\{1\})$ is a signed chain of $P$, and $p$ must satisfy the inequalities $-1\leq p_i +p_j \leq 1$.  Thus, either $p_i = p_j = 0$, $p_i = \pm 1$ and $p_j = 0$, $p_i= 0$  and $p_j = \pm 1$, or $p_i = - p_j = \pm 1$.  In the latter three cases, it is true that $\langle p, \alpha\rangle \neq 0$.  
\end{itemize} 
Thus $p$ satisfies all the properties of being an antichain of $P$.

Suppose $p$ is not an integer point of $\mathcal{C}(P)$.  Then there must be some chain 
\newline
$(\{c_1, \dots, c_m\}, \{s_1, \dots,s_{m-1}\})$ such that 
\[ p_{c_1} + s_1p_{c_2} + s_1s_2p_{c_3} + \dots + s_1s_2\dots s_{m-1}p_{c_m} \leq -2\]
\noindent or
\[ 2\leq p_{c_1} + s_1p_{c_2} + s_1s_2p_{c_3} + \dots + s_1s_2\dots s_{m-1}p_{c_m}. \]

This implies that for some $i,j \in [m]$, $2 \leq s_1\dots s_i p_i + s_1\dots s_jp_j$ or $ s_1\dots s_i p_i + s_1\dots s_jp_j \leq -2$.  From this, we can determine that $s_1\dots s_ip_i = s_1\dots s_jp_j = \pm 1$.  We have the following two cases:
\begin{itemize}
\item If $s_1\dots s_j = s_1\dots s_j$, then $p_i = p_j$. We also know that either $e_i - e_j \in P$ or $-e_i + e_j \in P$ from the transitivity of signed posets and the definition of signed chains.  Since  $\langle p, \pm e_i \mp e_j\rangle = 0$, we deduce that $p$ cannot be an antichain of $P$.
\item  If $s_1\dots s_j =- s_1\dots s_j$, then $p_i = - p_j$. We also know that
either $e_i + e_j \in P$ or $-e_i - e_j \in P$ from the transitivity of signed
posets and the definition of signed chains.  Since  $\langle p, \pm e_i
\pm e_j\rangle = 0$, we deduce that $p$ cannot be an antichain of $P$. \qedhere
\end{itemize}
\end{proof}

We note that there is no nice analogue for Theorem \ref{thm: order and chain}, since generally $\mathcal{C}_P$ and $\mathcal{O}_P$ are neither combinatorially equivalent nor Ehrhart equivalent.  One example of the latter is the example in which $P$ contains an element of the form $\pm e_i$.  Observe that $\mathcal{O}_P$ has no interior  lattice points, since the defining inequality $\pm x_i\geq 0$  prevents the origin from being an interior point and there are no other posibilities for an interior point of a polytope that is a subset of $[-1,1]^n$.  From the definition, we can see that $\mathcal{C}_P$ always has the origin as in interior point.  Thus in this case, these two polytopes cannot have the same Ehrhart polynomial.  
\bibliographystyle{amsplain}
\providecommand{\bysame}{\leavevmode\hbox to3em{\hrulefill}\thinspace}
\providecommand{\MR}{\relax\ifhmode\unskip\space\fi MR }
\providecommand{\MRhref}[2]{%
  \href{http://www.ams.org/mathscinet-getitem?mr=#1}{#2}
}
\providecommand{\href}[2]{#2}

\setlength{\parskip}{0cm} 

\end{document}